\documentclass[11pt]{amsproc}
 \usepackage[margin=1in]{geometry}
\usepackage{setspace,fullpage}
\usepackage{graphicx}
\usepackage[nice]{nicefrac}
\usepackage{amssymb}
\usepackage{multirow}
\usepackage{array}

\usepackage{amsmath,amsthm,amscd,amssymb, mathrsfs}

\usepackage{latexsym}

\numberwithin{equation}{section}

\theoremstyle{plain}
\newtheorem{theorem}{Theorem}[section]
\newtheorem{lemma}[theorem]{Lemma}
\newtheorem{corollary}[theorem]{Corollary}

\newtheorem{conjecture}[theorem]{Conjecture}
\newtheorem{question}[theorem]{Question}

\theoremstyle{definition}

\theoremstyle{remark}
\newtheorem{remark}[theorem]{Remark}

\newtheorem{case[theorem]}{Case}

\title[\parbox{14cm}{\centering{Extension theorems for Hamming varieties  \hspace{1in}}} \quad]{Extension theorems for Hamming varieties over finite fields}
\author{Daewoong Cheong, Doowon Koh,  and Thang Pham}
 
\address{Department of Mathematics\\
Chungbuk National University \\
Cheongju, Chungbuk 28644 Korea}
\email{daewoongc@chungbuk.ac.kr}

\address{Department of Mathematics\\
Chungbuk National University \\
Cheongju, Chungbuk 28644 Korea}
\email{koh131@chungbuk.ac.kr}

\address{Department of Mathematics\\
University of California, San Diego\\
La Jolla, CA  92093 USA}
\email{v9pham@ucsd.edu}

\thanks{
The first and second listed authors were supported by Basic Science Research Programs through National Research Foundation of Korea (NRF) funded by the Ministry of Education (NRF-2018R1D1A3B07045594 and NRF-2018R1D1A1B07044469, respectively). The third listed author was supported by Swiss National Science Foundation grant P400P2-183916..
}


\begin{document} 

\begin{abstract} We study the finite field extension estimates for Hamming varieties $H_j, j\in \mathbb F_q^*,$ defined by
$H_j=\{x\in \mathbb F_q^d: \prod_{k=1}^d x_k=j\},$
where $\mathbb F_q^d$ denotes the $d$-dimensional vector space over a finite field $\mathbb F_q$ with $q$ elements. We show that although the maximal Fourier decay bound on $H_j$ away from the origin is not good, 
the Stein-Tomas $L^2\to L^r$ extension estimate for $H_j$ holds.
\end{abstract}
\maketitle
\section{Introduction} 
The extension or restriction problem is one of central open questions in Euclidean harmonic analysis.
In 2002, Mockenhaupt and Tao \cite{MT04} initially studied this problem for algebraic varieties in the finite field setting.
Let $\mathbb F_q^d$ be the $d$-dimensional vector space over a finite field with $q$ elements. Throughout this paper, we assume that $q$ is an odd prime power.
Given complex-valued functions $f, g$ on $\mathbb F_q^d$ and $1\le s <\infty,$ we define
$$ \|g\|_{\ell^{s}(\mathbb F_q^d)} := \left(\sum_{m\in \mathbb F_q^d} |g(m)|^s\right)^{1/s} \quad \mbox{and}\quad \|f\|_{L^{s}(\mathbb F_q^d)} := \left(q^{-d}\sum_{x\in \mathbb F_q^d} |f(x)|^s\right)^{1/s}.$$
In addition, it is defined that $\|g\|_{\ell^{\infty}(\mathbb F_q^d)} := \max_{m\in \mathbb F_q^d} |g(m)|$ and $\|f\|_{L^{\infty}(\mathbb F_q^d)} := \max_{x\in \mathbb F_q^d} |f(x)|.$
The notation $\|g\|_{\ell^{s}(\mathbb F_q^d)}$ indicates that  the function $g$ is defined on the space $\mathbb F_q^d$ with counting measure. On the other hand, the notation $\|f\|_{L^{s}(\mathbb F_q^d)}$ 
tells us that the function $f$ is defined on the space $\mathbb F_q^d$ with normalized counting measure.
Let $V$ be an algebraic variety in $\mathbb F_q^d.$ We endow $V$ with a normalized surface measure $d\sigma$ which means that  the mass of each point of $V$ is $1/|V|,$ where 
$|V|$ denotes the cardinality of the set $V.$ For a function $f:V \to \mathbb C$ and $1\le s<\infty,$ we define
$$ \|f\|_{L^s(V, d\sigma)}:= \left(\frac{1}{|V|} \sum_{x\in V} |f(x)|^s\right)^{1/s}.$$
We also define $\|f\|_{L^\infty(V, d\sigma)}:=\max_{x\in V} |f(x)|.$\\

 The Fourier transform of $g$, denoted by $\widehat{g},$ is defined by
$$ \widehat{g}(x)= \sum_{m\in \mathbb F_q^d} \chi(-m\cdot x) g(m),$$
where $\chi$ denotes the canonical additive character of $\mathbb F_q$, and $m\cdot x$ is the usual dot-product of $m$ and $x.$ We recall that the orthogonality of $\chi$ states that 
$$ \sum_{\alpha \in \mathbb F_q^d} \chi(n\cdot \alpha) =\left\{\begin{array}{ll}  0 \quad &\mbox {if} \quad n\ne (0,\ldots, 0) \\
 q^d \quad &\mbox{if} \quad n=(0,\ldots,0) . \end{array}\right.$$

The inverse Fourier transform of $f$, denoted by $f^{\vee}$, is defined by
$$ f^\vee(m):=q^{-d}\sum_{x\in \mathbb F_q^d} \chi(m\cdot x) f(x).$$ Furthermore,  the inverse Fourier transform of the measure $fd\sigma$ is given by 
$$ (fd\sigma)^\vee(m):= \frac{1}{|V|} \sum_{x\in V} \chi(m\cdot x) f(x).$$
We denote by $R^*(p\to r)$ the smallest constant such that the following extension estimate
\begin{equation}\label{Extension} \|(fd\sigma)^\vee\|_{\ell^r(\mathbb F_q^d)} \le R^*(p\to r) \|f\|_{L^p(V, d\sigma)}\end{equation}
holds for all functions $f$ on $V.$ Note that $R_V^*(p\to r)$ may depend on $q$, the size of the underlying finite field $\mathbb F_q.$
The extension problem for the variety $V$ is to determine all exponents $1\le p, r\le \infty$ such that $R_V^*(p\to r)$ is independent of $q.$
For $A, B>0$, we will write $A\lesssim B$ if $A\le C B$ for some constant $C>0$ independent of $q.$ We will also use $A\sim B$ if $A\lesssim B$ and $B\lesssim A.$
By a well-known duality,  the inequality \eqref{Extension} is the same as the following restriction estimate:
$$ \|\widehat{g}\|_{L^{p'}(V, d\sigma)} \le R_V^*(p\to r) \|g\|_{\ell^{r'}(\mathbb F_q^d)},$$
where $p', r'$ denote the dual exponents of $p,r$, respectively (i.e.  $1/p+1/p'=1$ and $1/r+1/r'=1$).\\

When $|V|\sim q^{d-1}$, necessary conditions for $R^*_V(p\to r)$ bound can be obtained from the size of a maximal affine subspace lying on $V.$ Indeed, Mockenhaupt and Tao \cite{MT04} showed that  if $|V|\sim q^{d-1}$ and $V$ contains an affine subspace $H$ with $|H|=q^k,$ then  necessary conditions for $R^*_V(p\to r)$ bound  are given by 
$$\label{Necessary2}
r\geq \frac{2d}{d-1}  \quad \mbox{and} \quad r\geq\frac{p(d-k)}{(p-1)(d-1-k)}.$$
In dimension two, the extension problem for algebraic curves $V$ was completely solved by Shen and the second listed author \cite{KS12} who showed that the above necessary conditions are also sufficient conditions for $R^*_V(p\to r)$ bound. For this reason, we will restrict ourselves to the case when $d\ge 3.$ 
In particular, we have the following conjecture for $R^*_V(2\to r)$ bound.
\begin{conjecture}\label{conj1} Let $V$ be an algebraic variety in $\mathbb F_q^d.$ Suppose that $|V|\sim q^{d-1}$ and $V$ contains an affine subspace $H$ with $|H|=q^k.$ Then we have
$$ R^*_V(2\to r) \lesssim 1 \quad \mbox{if} \quad \frac{2(d-k)}{d-k-1}\le r\le \infty.$$
\end{conjecture}
By the norm nesting property (see Section \ref{sec2}), one can check that if $1\le r_1\le r_2\le \infty,$ then $R^*_V(2\to r_2) \le R^*_V(2\to r_1).$ This implies that a smaller exponent gives a better result on the restriction problem.
Thus if we want to  establish the sharp $R_V^*(2\to r)$ bound, then we only needs to find the smallest exponent $r$ such that $R_V^*(2\to r)\lesssim 1.$ Namely, to confirm Conjecture \ref{conj1} it suffices to prove that
$$ R^*_V\left(2\to \frac{2(d-k)}{d-k-1}\right)\lesssim 1.$$\\

The finite field extension problem has been studied only for few algebraic varieties with relatively simple structures such as spheres, paraboloids, or cones. 
For example, Mockenhaupt and Tao \cite{MT04} addressed  results on the problem for paraboloids and cones, and their work for those varieties has been recently improved by other researchers 
(see \cite{KS, IK09, LL10,  Le13, Ko16, IKL17, RS18, KPV, Le19}).
For spheres, Iosevich and the second listed author \cite{IK10}  obtained nontrivial results which have been improved in the papers \cite{IKLST, KPV}. While several new methods have been used in studying the Euclidean extension problem, there are only few known skills to deduce the results on the finite field extension problem. Among other things, the Stein-Tomas argument can be applied in the finite field case to deduce $R^*_V(2\to r)$ bound. Indeed, Mockenhaupt and Tao \cite{MT04} introduced the finite field Stein-Tomas argument. In particular, we have the following lemma which is a special case of Lemma 6.1 in \cite{MT04}.
\begin{lemma}[The finite field Stein-Tomas argument] \label{ST}
Let $d\sigma$ be the normalized surface measure on an algebraic variety $V$ in $\mathbb F_q^d.$ Suppose that
\begin{equation}\label{C1} |V|\sim q^{d-1}\end{equation}
and 
\begin{equation}\label{C2} \max_{m\in \mathbb F_q^d \setminus \{(0,\ldots,0)\}} |(d\sigma)^\vee(m)| \lesssim q^{-\frac{\alpha}{2}}\end{equation}
for some $ \alpha >0.$ Then we have
$$ R_V^*\left( 2\to \frac{2(\alpha+2)}{\alpha}\right) \lesssim 1.$$
\end{lemma}

For a general version of Lemma \ref{ST}, we refer readers to \cite{CC}. To apply Lemma \ref{ST}, one needs to compute the maximal Fourier decay bound on the measure $d\sigma$ away from the origin. For example,  when $V$ is a sphere or a paraboloid, it is well-known that one can take $\alpha=(d-1)/2$ in \eqref{C2}, and thus 
$R^*_V(2\to (2d+2)/(d-1))\lesssim 1$ (see \cite{MT04, IK10}). This result is called as the Stein-Tomas result which gives the optimal $R^*_V(2\to r)$ bound in general. 
Now, we pose an interesting question. 
\begin{question}\label{Q1} Does there exist a variety $V$ such that the condition \eqref{C2} does not hold with $\alpha=(d-1)/2$ but we still have the Stein-Tomas result for $V$?\end{question}
There exist several varieties $V$ in $\mathbb F_q^d$ such that the condition \eqref{C2} does not hold with $\alpha=(d-1)/2$ and the Stein-Tomas result can not be obtained. For example, if $d$ is even and $d\sigma$ is the normalized surface measure on the variety $V:=\{x\in \mathbb F_q^d: x_1^2-x_2^2+\cdots+ x_{d-1}^2-x_d^2=0\}$, then 
$$ \max_{m\in \mathbb F_q^d \setminus \{(0,\ldots,0)\}} |(d\sigma)^\vee(m)|\sim q^{-\frac{(d-2)}{2}}$$
and  the optimal $L^2\to L^r$ extension estimate for $V$ is that $R^*_V(2\to 2d/(d-2)) \lesssim 1,$ which is much weaker than the Stein-Tomas result (see Theorem 2.1 and Lemma 4.1 in \cite{KS}). \\

Our main purpose of this paper is to provide a concrete variety which gives a positive answer to the above question.\\

For each $j\in \mathbb F_q^*$, the Hamming variety $H_j$ in $\mathbb F_q^d$ is defined by
\begin{equation}\label{defH} H_j=\left\{x=(x_1,\ldots, x_d)\in \mathbb F_q^d: \prod_{k=1}^d x_k=j\right\}.\end{equation}
Since $j\ne 0$, it is not hard to see $|H_j|=(q-1)^{d-1}\sim q^{d-1}.$ Our main result is as follows.
\begin{theorem}\label{main}
Let $d\sigma_j$ denote the normalized surface measure on the Hamming variety $H_j$ in $\mathbb F_q^d$ defined as in \eqref{defH}.
Then, for every $j\ne 0$ and $d\ge 3,$ we have
$$ R^*_{H_j} \left(2\to \frac{2d+2}{d-1}\right) \lesssim 1.$$
\end{theorem}

It is not hard to see that the Hamming variety $H_j$ with $j\ne 0$ does not contain any line. Taking $k=0$  
in Conjecture \ref{conj1} we may conjecture that 
$$ R^*_{H_j}(2\to r) \lesssim 1 \quad \mbox{if} \quad \frac{2d}{d-1}\le r\le \infty.$$
Theorem \ref{main} is much weaker than this conjecture, but for $d\ge 4$ it can not be obtained by simply applying Lemma \ref{ST}. To see this, notice from Corollary \ref{CorDecaysigma} in Section \ref{sec3} that 
 $$ \max_{m\in \mathbb F_q^d\setminus\{(0,\ldots, 0)\}}|(d\sigma_j)^\vee(m)|\sim   q^{-1}\quad \mbox{for} \quad d\ge 4. $$
Combining this with Lemma \ref{ST}, we only see that $R^*_{H_j}(2\to 4)\lesssim 1$ which is much weaker than Theorem \ref{main} for $d\ge 4.$ To prove Theorem \ref{main}, we decompose the surface measure on $H_j$ into $(d+1)$ surface measures and each of them will be analyzed.
\begin{remark} In even dimensions, much progress on extension problems for the paraboloid $P$ in $\mathbb F_q^d$ has been made by improving the additive energy estimate for subsets of the paraboloid (see, for example, \cite{Le13, IKL17, RS18}).  
Recall that for a set $E$ in $\mathbb F_q^d$, the additive energy of the set $E$, denoted by $\Lambda(E)$, is defined by 
$$ \Lambda(E)=\sum_{x,y,z,w\in E: x+y=z+w} 1.$$ 
When a set $E$ lies on the Hamming variety $H_j$, it seems that it is a challenging problem to obtain a good upper bound of $\Lambda(E).$
\end{remark}

\section{Discrete Fourier analysis} \label{sec2}
In this section, we review the discrete Fourier analysis which will be our main tool in proving our main result. The proofs of all statements in this section can be found in Been Green's lecture note \cite{G}.
In the finite field setting, the norm nesting properties hold: for $1\le p_1\le p_2\le \infty,$
\begin{equation} \label{nesting}\|g\|_{\ell^{p_2}(\mathbb F_q^d)} \le  \|g\|_{\ell^{p_1}(\mathbb F_q^d)}\end{equation}
and
$$  \|f\|_{L^{p_1}(\mathbb F_q)} \le  \|f\|_{L^{p_2}(\mathbb F_q^d)}, \quad\|f\|_{L^{p_1}(V, d\sigma)} \le  \|f\|_{L^{p_2}(V, d\sigma)}.$$
The Plancherel theorem states that
$$ \|\widehat{g}\|_{L^2(\mathbb F_q^d)}=\|g\|_{\ell^2(\mathbb F_q^d)} \quad \mbox{or} \quad
\|f^\vee\|_{\ell^2(\mathbb F_q^d)}=\|f\|_{L^2(\mathbb F_q^d)}$$
which can be easily deduced by the orthogonality of $\chi.$ We also note that $ \widehat{(f^\vee)}=f.$
Given functions $g_1, g_2: \mathbb F_q^d \to \mathbb C,$  the convolution function of $g_1$ and $g_2$, denoted by $g_1\ast g_2$, is defined by
$$ g_1\ast g_2 (m)=\sum_{n\in \mathbb F_q^d} g_1(m-n) g_2(n).$$
One can easily check that $\widehat{g_1\ast g_2}=\widehat{g_1} \widehat{g_2}.$ We recall that Young's inequality for convolutions states that  if $ 1\le a,b, r \le \infty$ satisfy $1/r=1/a+1/b-1,$ then
$$ \|g_1\ast g_2\|_{\ell^r(\mathbb F_q^d)} \le \|g_1\|_{\ell^a(\mathbb F_q^d)} \|g_2\|_{\ell^b(\mathbb F_q^d)} .$$
We will invoke the following well-known interpolation theorem.
\begin{theorem} [Riesz-Thorin]\label{tmR} Let $1\le p_0, p_1, r_0, r_1\le \infty$ with $p_0\le p_1$ and $r_0\le r_1.$
Suppose that $T$ is a linear operator and the following two estimates hold for all functions $g$ on $\mathbb F_q^d:$ 
$$ \|Tg\|_{\ell^{r_0}(\mathbb F_q^d)} \le M_0 \|g\|_{\ell^{p_0}(\mathbb F_q^d)}\quad \mbox{and} \quad
\|Tg\|_{\ell^{r_1}(\mathbb F_q^d)} \le M_1 \|g\|_{\ell^{p_1}(\mathbb F_q^d)}.$$
Then we have
$$\|Tg\|_{\ell^{r}(\mathbb F_q^d)} \le M_0^{1-\theta} M_1^{\theta} \|g\|_{\ell^{p}(\mathbb F_q^d)}$$
for any $0\le \theta \le 1$ with
$$ \frac{1-\theta}{r_0} + \frac{\theta}{r_1}=\frac{1}{r} \quad \mbox{and}\quad
\frac{1-\theta}{p_0} + \frac{\theta}{p_1}=\frac{1}{p}.$$
\end{theorem}

\section{Fourier decay on Hamming varieties}\label{sec3}
Recall that $d\sigma_j$ denotes the normalized surface measure on the Hamming variety $H_j$ in $\mathbb F_q^d.$ 
In this section, we introduce an explicit form of $(d\sigma_j)^\vee$ which makes a crucial role in proving Theorem \ref{main}. 

\begin{lemma}\label{Decaysigma} For each $j\in \mathbb F_q^*$, let $d\sigma_j$ be the normalized surface measure on the Hamming variety $H_j$ in $\mathbb F_q^d.$ For each $m\in \mathbb F_q^d,$ denote by $\ell_m$ the number of zero components of $m.$ Then we have
\begin{equation}\label{CK1} (d\sigma_j)^\vee (m)= (-1)^{d-\ell_m} ~(q-1)^{-(d-\ell_m)} \quad\mbox{if}\quad 1\le \ell_m \le d.\end{equation}
In addition, if $\ell_m =0$, then $|(d\sigma_j)^\vee (m)|\lesssim q^{-\frac{(d-1)}{2}}.$
\end{lemma}
\begin{proof} Since $j\ne 0,$ we  see that $|H_j|=(q-1)^{d-1}\sim q^{d-1}$ and all components of any element in the Hamming variety $H_j$ are not zero.
By definition, it follows
$$(d\sigma_j)^\vee (m) =\frac{1}{|H_j|} \sum_{x\in H_j} \chi(m\cdot x)= \frac{1}{|H_j|} \sum_{x_1, x_2,\ldots, x_d\in \mathbb F_q^*: x_1x_2\cdots x_d=j} \chi(m\cdot x).$$
\noindent \textbf{Case 1.} Assume that $\ell_m=d.$
Then $m=(0,\ldots,0)$ and so $(d\sigma)^\vee (m)=1.$\\

\noindent \textbf{Case 2.} Assume that $\ell_m= k$ for some $k=1, 2, \ldots, (d-1).$  Without loss of generality, we may assume that $m_1=m_2=\cdots =m_k=0,$ and $ m_i\ne 0$ for $ k+1\le i \le d.$ It follows that
\begin{align*} (d\sigma_j)^\vee (m) =& \frac{1}{|H_j|} \sum_{x_1, x_2,\ldots, x_{d-1} \in \mathbb F_q^*}  \chi(m_1x_1+ m_2x_2+ \cdots + m_{d-1}x_{d-1})~ \chi\left(\frac{jm_d}{x_1x_2 \cdots x_{d-1}}\right)\\
=& \frac{1}{|H_j|} \sum_{x_1, x_2,\ldots, x_{d-1} \in \mathbb F_q^*} \chi(m_{k+1}x_{k+1} + \cdots + m_{d-1} x_{d-1}) ~ \chi\left(\frac{jm_d}{x_1x_2 \cdots x_{d-1}}\right),
\end{align*}
where we assume that if $k=d-1,$ then $\chi(m_{k+1}x_{k+1} + \cdots + m_{d-1} x_{d-1})=1.$
Since $j, m_d \ne 0$, we see from the orthogonality of $\chi$ that for each $x_2, \cdots, x_{d-1} \in \mathbb F_q^*,$
$$ \sum_{x_1\in\mathbb F_q^*} \chi\left(\frac{jm_d}{x_1x_2 \cdots x_{d-1}}\right) =-1.$$
Therefore we have
$$ (d\sigma_j)^\vee (m) = - \frac{1}{|H_j|} \sum_{x_{k+1}, \ldots, x_{d-1}\in \mathbb F_q^*} \chi(m_{k+1}x_{k+1} + \cdots + m_{d-1} x_{d-1}) ~\left(\sum_{x_2, \ldots, x_{k} \in \mathbb F_q^*} 1 \right).$$
Since $|H_j|=(q-1)^{d-1}$ and $ m_{k+1}, \ldots, m_{d-1} \ne 0$,  we conclude from the orthogonality of $\chi$ that
$$ (d\sigma_j)^\vee (m)= (-1)^{d-k} ~(q-1)^{-(d-k)} \quad\mbox{if}\quad 1\le \ell_m =k \le d-1,$$
which completes the proof in the case when $ 1\le \ell_m \le (d-1).$\\

\noindent \textbf{Case 3.} Assume that $\ell_m =0.$ Then all components of $m$ are not zero.
As in Case 2, we can write
$$(d\sigma_j)^\vee (m) =\frac{1}{|H_j|} \sum_{x_1, x_2,\ldots, x_{d-1} \in \mathbb F_q^*}  \chi(m_1x_1+ m_2x_2+ \cdots + m_{d-1}x_{d-1})~ \chi\left(\frac{jm_d}{x_1x_2 \cdots x_{d-1}}\right).$$
Since $|H_j|\sim q^{d-1},$ the last part of the theorem is a direct consequence from the following theorem due to  Deligne \cite{De}:

\begin{theorem} [Multiple Kloosterman sums] For $a_1, a_2, \ldots, a_s, b\in \mathbb F_q^*$, we have
$$\left| \sum_{x_1,x_2, \ldots, x_s\in \mathbb F_q^*} \chi(a_1x_1 +\cdots+ a_s x_s + b x_1^{-1} x_2^{-1} \cdots x_s^{-1}) \right| \le (s+1) q^{\frac{s}{2}}.$$\end{theorem}
To find further references for Multiple Kloosterman sums, we refer readers to [P.254, \cite{LN97}].
\end{proof}

\bigskip
The following result follows immediately from Lemma \ref{Decaysigma}.
\begin{corollary} \label{CorDecaysigma} For each $j\in \mathbb F_q^*$, let $d\sigma_j$ denote the normalized surface measure on $H_j$ in $\mathbb F_q^d.$
Then we have
$$ \max_{m\in \mathbb F_q^d\setminus\{(0,\ldots, 0)\}}|(d\sigma)^\vee(m)|\sim   q^{-1}\quad \mbox{for} \quad d\ge 3. $$
\end{corollary}

\section{Proof of Theorem \ref{main}}

We aim to prove that the extension estimate
$$ \|(fd\sigma_j)^\vee\|_{\ell^{\frac{2d+2}{d-1}}(\mathbb F_q^d)} \lesssim \|f\|_{L^2(H_j, d\sigma_j)}$$
holds for all complex-valued functions $f$ on $H_j.$
By duality, it suffices to prove that the restriction estimate
$$ \|\widehat{g}\|_{L^2(H_j, d\sigma_j)} \lesssim \|g\|_{\ell^{\frac{2d+2}{d+3}}(\mathbb F_q^d)}$$
holds for all complex-valued functions $g$ on $\mathbb F_q^d.$
By the $RR^*$ method (see \cite{G}), we see that
$$ \|\widehat{g}\|^2_{L^2(H_j, d\sigma_j)} = \left<g,~ g\ast (d\sigma_j)^\vee\right>_{\ell^2(\mathbb F_q^d)}.$$
Here, we recall that if $g, h: \mathbb F_q^d\to \mathbb C,$ then 
$$ \left<g,~ h \right>_{\ell^2(\mathbb F_q^d)} := \sum_{m\in \mathbb F_q^d} g(m) \overline{h} (m).$$
For each $k=0,1,\ldots, d,$  define
$$ N_k=\{m\in \mathbb F_q^d: k~\mbox{components of}~m~\mbox{are exactly zero}\}.$$
We decompose $(d\sigma_j)^\vee$ as
$$ (d\sigma_j)^\vee(m)= ((d\sigma_j)^\vee 1_{N_0})(m) + \sum_{k=1}^d ((d\sigma_j)^\vee 1_{N_k})(m).$$
It follows that
\begin{align*} \|\widehat{g}\|^2_{L^2(H_j, d\sigma_j)} =& \left<g,~ g\ast ((d\sigma_j)^\vee 1_{N_0})\right>_{\ell^2(\mathbb F_q^d)} + \left<g,~ g\ast \sum_{k=1}^d((d\sigma_j)^\vee 1_{N_k})\right>_{\ell^2(\mathbb F_q^d)}\\
=&\left<g,~ g\ast ((d\sigma_j)^\vee 1_{N_0})\right>_{\ell^2(\mathbb F_q^d)} + \sum_{k=1}^d\left<g,~ g\ast ((d\sigma_j)^\vee 1_{N_k})\right>_{\ell^2(\mathbb F_q^d)}.\end{align*}
Hence, to complete the proof, it will be enough to show that the following two inequalities hold for all functions $g: \mathbb F_q^d\to \mathbb C$ and for all $k=1,2,\ldots, d:$
\begin{equation}\label{pt1}
\left<g,~ g\ast ((d\sigma_j)^\vee 1_{N_0})\right>_{\ell^2(\mathbb F_q^d)} \lesssim \|g\|^2_{\ell^{\frac{2d+2}{d+3}}(\mathbb F_q^d)}\end{equation}
and
\begin{equation}\label{pt2}\left<g,~ g\ast ((d\sigma_j)^\vee 1_{N_k})\right>_{\ell^2(\mathbb F_q^d)} \lesssim 
\|g\|^2_{\ell^{\frac{2d+2}{d+3}}(\mathbb F_q^d)}.
\end{equation}
In the following subsections, we will give the proofs of inequalities \eqref{pt1} and \eqref{pt2}, which completes  the proof of Theorem \ref{main}.
\subsection{Proof of inequality \eqref{pt1}}
By H\"{o}lder's inequality, we have
\[\left<g,~ g\ast ((d\sigma_j)^\vee 1_{N_0})\right>_{\ell^2(\mathbb F_q^d)} \lesssim \|g\|_{\ell^{\frac{2d+2}{d+3}}(\mathbb F_q^d)}  \| g\ast ((d\sigma_j)^\vee 1_{N_0})\|_{\ell^{\frac{2d+2}{d-1}}(\mathbb{F}_q^d)}.\]
Thus, in order to prove the inequality (\ref{pt1}), it is enough to prove the following:
\[\|g\ast ((d\sigma_j)^\vee 1_{N_0})\|_{\ell^{\frac{2d+2}{d-1}}(\mathbb{F}_q^d)}\lesssim \|g\|_{\ell^{\frac{2d+2}{d+3}}(\mathbb{F}_q^d)}.\]
By the Riesz-Thorin interpolation theorem (Theorem \ref{tmR}), it suffices to prove the following two inequalities:
\begin{equation}\label{eq1}
\|g\ast ((d\sigma_j)^\vee 1_{N_0})\|_{\ell^{2}(\mathbb{F}_q^d)} \lesssim q\|g\|_{\ell^{2}(\mathbb F_q^d)},
\end{equation}
and
\begin{equation}\label{eq2}
\|g\ast ((d\sigma_j)^\vee 1_{N_0})\|_{\ell^{\infty}(\mathbb{F}_q^d)} \lesssim q^{-\frac{d-1}{2}}\|g\|_{\ell^{1}(\mathbb F_q^d)}.
\end{equation}
For the first inequality \eqref{eq1}, using the Plancherel theorem gives us the following estimate: 
\begin{align}\label{EE}\|g\ast ((d\sigma_j)^\vee 1_{N_0})\|_{\ell^{2}(\mathbb{F}_q^d)}&=\|\widehat{g} ~\widehat{ ((d\sigma_j)^\vee 1_{N_0}})\|_{L^2(\mathbb F_q^d)}\\
&\le \left[\max_{x\in \mathbb F_q^d} \left|\widehat{ ((d\sigma_j)^\vee 1_{N_0}}) (x)\right|\right] \|g\|_{\ell^2(\mathbb F_q^d)}. \nonumber \end{align}
Now, for each fixed $x\in\mathbb F_q^d,$ we have
$$ \left|\widehat{ ((d\sigma_j)^\vee 1_{N_0}}) (x)\right|=\left|\sum_{m\in N_0} \chi(-m\cdot x) \frac{1}{|H_j|} \sum_{y\in H_j} \chi(m\cdot y)\right|= \frac{1}{|H_j|} \left|\sum_{y\in H_j} \sum_{m\in N_0} \chi(m\cdot (y-x))\right|.$$
Using a change of variable by letting $z=y-x,$ and the triangle inequality, 
 \begin{equation}\label{sam}\left|\widehat{ ((d\sigma_j)^\vee 1_{N_0}}) (x)\right|\le  \frac{1}{|H_j|} \sum_{z\in H_j-x} \left|\sum_{m\in N_0} \chi(m\cdot z)\right|\lesssim \frac{1}{q^{d-1}} \sum_{z\in \mathbb F_q^d}\left|\sum_{m\in N_0} \chi(m\cdot z)\right|.\end{equation}
We decompose $\sum_{z\in \mathbb F_q^d}$ into $\sum_{k=0}^d \sum_{z\in N_k}$, and use the orthogonality of $\chi.$ Then we obtain
\begin{align*}\left|\widehat{ ((d\sigma_j)^\vee 1_{N_0}}) (x)\right|&\lesssim  \frac{1}{q^{d-1}}\sum_{k=0}^d\sum_{z\in N_k}\left\vert \sum_{m_1, \ldots, m_d\in \mathbb{F}_q^*}\chi(m\cdot z)\right\vert\\
&= \frac{1}{q^{d-1}}\sum_{k=0}^d (q-1)^k|N_k|\lesssim q,\end{align*}
where we also used the fact that $|N_k|={d\choose k} q^{d-k}\sim q^{d-k}.$ Thus the inequality \eqref{eq1} holds.

The second inequality \eqref{eq2} follows by using Young's inequality for convolutions and  the second part of Lemma \ref{Decaysigma}. More precisely, we have
\[\|g\ast ((d\sigma_j)^\vee 1_{N_0})\|_{\ell^{\infty}(\mathbb{F}_q^d)} \le \|((d\sigma_j)^\vee 1_{N_0}) \|_{\ell^{\infty}(\mathbb{F}_q^d)}\|g\|_{\ell^{1}(\mathbb F_q^d)} \lesssim q^{-\frac{d-1}{2}}\|g\|_{\ell^{1}(\mathbb F_q^d)}.\]
Hence, the proof of the inequality \eqref{pt1} is complete. 

\subsection{Proof of inequality \eqref{pt2}} We will prove  much better inequality than the inequality \eqref{pt2}. Notice from the norm nesting property \eqref{nesting} that 
$$ \|g\|_{\ell^{2}(\mathbb F_q^d)}\le \|g\|_{\ell^{\frac{2d+2}{d+3}}(\mathbb F_q^d)}.$$
To complete the proof of the inequality \eqref{pt2}, it will be enough to show that for each $k=1,2,\ldots,d,$
$$\left<g,~ g\ast ((d\sigma_j)^\vee 1_{N_k})\right>_{\ell^2(\mathbb F_q^d)} \lesssim 
\|g\|^2_{\ell^{2}(\mathbb F_q^d)}.$$
By H\"{o}lder's inequality, we have
$$\left<g,~ g\ast ((d\sigma_j)^\vee 1_{N_k})\right>_{\ell^2(\mathbb F_q^d)} \le \|g\|_{\ell^{2}(\mathbb F_q^d )}~\|g\ast ((d\sigma_j)^\vee 1_{N_k})\|_{\ell^{2}(\mathbb{F}_q^d)}.$$

As seen in \eqref{EE}, it suffices by  the Plancherel theorem  to prove that
\begin{equation}\label{goa}\left[\max_{x\in \mathbb F_q^d} \left|\widehat{ ((d\sigma_j)^\vee 1_{N_k}}) (x)\right|\right] \lesssim 1.\end{equation}

Fix $x\in \mathbb F_q^d$ and $k \in \{1,2, \ldots, d\}.$ From \eqref{CK1} of Lemma \ref{Decaysigma}, we see that for each $m\in \mathbb F_q^d,$
$$ ((d\sigma_j)^\vee 1_{N_k}) (m)= (-1)^{d-k}(q-1)^{-d+k} 1_{N_k}(m).$$
Therefore, we have
$$ \left|\widehat{ ((d\sigma_j)^\vee 1_{N_k}}) (x)\right|=(q-1)^{-d+k} \left|\widehat{N_k}(x)\right|.$$
Since $|\widehat{N_k}(x)|=|\sum_{m\in N_k} \chi(-m\cdot x)| \le |N_k|= {d\choose{k}} q^{d-k}\sim q^{d-k},$ we conclude that
$$\left|\widehat{ ((d\sigma_j)^\vee 1_{N_k}}) (x)\right|\lesssim 1.$$
Thus, the inequality \eqref{goa} holds, which  completes the proof of the inequality \eqref{pt2}.

\bibliographystyle{amsplain}

\end{document}